\documentclass{article}
\usepackage{amssymb}
\usepackage{amsmath}
\usepackage{amsthm}
\numberwithin{equation}{section}

\newtheorem{theorem}{Theorem}[section]

\newtheorem{lemma}[theorem]{Lemma}
\newtheorem{corollary}[theorem]{Corollary}

\newtheorem{example}[theorem]{Example} 
 
\newcommand{\norm}{\left\Vert\,\cdot\,\right\Vert}

\newcommand{\C}{\mathbb C}
\newcommand{\D}{\mathbb D}

\newcommand{\M}{\mathbb M}
\newcommand{\T}{\mathbb T}

\newcommand{\R}{\mathbb R}
\newcommand{\N}{\mathbb N}
\newcommand{\Z}{\mathbb Z}
\newcommand{\I}{\mathbb I}

\newcommand{\B}{\mathcal B}
\newcommand{\lv}{\left\vert}
\newcommand{\rv}{\right\vert}
\newcommand{\lV}{\left\Vert}
\newcommand{\rV}{\right\Vert}
\newcommand{\s}{\smallskip}
\newcommand{\codim}{{\rm codim}\,}
\def \CBA{commutative Banach algebra}
\def\CG{countably generated}

\def\obox {\hbox{\vrule\vbox{\hrule
 \hbox spread 4pt{\hfil\vbox spread 6pt{\vfil}\hfil}
\hrule}\vrule}}

\def\qed{{\unskip\nobreak\hfil\penalty50\hskip1em\nobreak\hfil 
\obox
\parfillskip=0pt\finalhyphendemerits=0\par}\vskip3mm}

\title
{Generators of maximal left ideals in  Banach algebras} 

 \author{H.\ G.\ Dales, and W.\ \.Zelazko}

\begin{document}
\maketitle
 
 \begin{abstract}{In 1971, Grauert and Remmert proved that a commutative, complex, Noetherian Banach algebra is necessarily finite-dimensional. 
More precisely, they proved that a commutative, complex   Banach algebra has finite dimension over $\C$ whenever all the closed 
 ideals in the algebra are (algebraically) finitely generated. In 1974, Sinclair and Tullo obtained a non-commutative version of this result. 
In 1978, Ferreira and Tomassini improved the result of Grauert and Remmert  by showing that the statement is also true if one replaces `closed ideals' 
by `maximal ideals in the {\v S}ilov boundary of $A$'.  
We shall give a shorter proof of this latter result, together with some extensions and related examples.}

(2010) Subject classification: Primary 46H10;  Secondary 46J10
  \end{abstract}


\section{Introduction} 

\subsection{Notation} The natural numbers and the integers are $\N$ and $\Z$, respectively. For $n \in \N$,
 we set $\N_n=\{1,\dots,n\}$.  The unit circle and open unit disc in the complex field $\C$ are $\T$ and $\D$, respectively; the real line is  $\R$, 
and $\R^+= \{s \in\R : s\geq 0\}$. The algebra of all $n\times n$ matrices over $\C$ is denoted by $\M_{\,n}$; the matrix units in the matrix algebra 
${\mathbb M}_{\,n}$  are denoted by $E_{i,j}$ for $i,j\in \N_n$; the identity matrix is denoted by $\iota_n$.
 
 We write $c_{\,0}$ and $\ell^{\,p}$ (for $p\in [1,\infty])$ for the standard sequence spaces on $\N$; we write $c =c_{\,0}\oplus \C 1$ for the Banach 
space of all convergent sequences, where $1 =(1,1,1,\dots)$. 

Let $A$ be an (associative) algebra, always taken to be over the complex field.    
In the case where $A$ does not have an identity, the algebra formed by adjoining an identity to $A$ is denoted by $A^{\sharp}$; 
we take $A^{\sharp}$ to be $A$ in the case where $A$ already has an identity.   A linear subspace $I$ of $A$ is a {\it left ideal\/} if $AI\subset I$,
 a {\it right  ideal\/} if $IA\subset I$, 
 and an {\it ideal\/} if $AI+IA \subset I$.  The {\it Jacobson radical\/} of an algebra $A$ is denoted by $J(A)$; 
$J(A)$ is the intersection of the  maximal modular left (or right) ideals of $A$, and it is an ideal in $A$.

Let  $S$  be a subset of $A$. Then the {\it left ideal generated by $S$\/} is the intersection of the left ideals of $A$ that contain $S$; this left ideal 
is denoted by $\langle S \rangle$. Clearly,   
$$ \langle S \rangle =\left\{\sum_{i=1}^n a_is_i : a_1,\dots,a_n \in  A^{\sharp}, \,s_1,\dots,s_n \in S,\,n\in\N\right\}\,.
$$
 The left ideal generated by a finite subset $\{a_1,\dots,a_n\}$ is equal to
\begin{equation}\label{(1.1)}
I= A^{\sharp}a_1 +\cdots+A^{\sharp}a_n \,,
\end{equation}
and it is denoted by $\langle a_1,\dots,a_n\rangle$. 

Let  $I$ be a left ideal in the algebra $A$. Then:   $I$ is {\it countably generated\/}  if there is a countable set $S$
 with $I = \langle S \rangle$; $I$ 
is {\it finitely generated\/} if there are elements $a_1,\dots,a_n$ in $A$ 
with  $I=\langle a_1,\dots,a_n\rangle$, and  in this case,  $a_1,\dots,a_n$  are {\it generators\/} of $I$;  $I$ is {\it singly generated\/} (by $a$) if 
$I= \langle a\rangle$ for some $a\in A$.  
 
An algebra $A$ is {\it left Noetherian\/} if the family of left ideals in $A$ satisfies the ascending chain condition; 
this is the case if and only if each left ideal in $A$ is  finitely generated.  See \cite[Chapter VIII]{Hun}, for example.\medskip

\subsection{The families  ${\frak I}_\infty (A)$ and ${\frak U}_\infty (A)$} Now suppose that $A$ is an algebra. 
We  denote by ${\frak I}_\infty (A)$ 
the family of all    left ideals in $A$ which are not finitely generated, and by ${\frak U}_\infty (A)$ the family of all left ideals in $A$ which are not {\CG}, so that
  ${\frak U}_\infty (A)\subset {\frak I}_\infty (A)$.

The  result of Grauert and Remmert which was stated  in the abstract 
 can be formulated as follows. See  the Appendix to $\S5$ in \cite{GRe}. \vspace{-\baselineskip}\s

\begin{theorem} \label{1.1}
 Let $A$ be a  commutative Banach algebra. Suppose that every closed ideal in $A$ is finitely generated. 
Then $A$ is finite dimensional.\qed
\end{theorem}\s
 
 This theorem was  generalized by Sinclair and Tullo \cite{ST} to the non-commutative case; we state their result as follows.\s
 
 \begin{theorem} \label{1.2}
Let $A$ be a Banach algebra. Suppose that every closed left ideal in $A$ is finitely generated. 
Then $A$ is finite dimensional.\qed
\end{theorem}\s

In the proof of their result, Grauert and Remmert used the following fact \cite[Bemerkung 2, p.~54]{GRe}.
 Let  $I$ be an ideal in a {\CBA} $A$, and suppose that the closure ${\overline I}$ of $I$ is
finitely generated. Then $I$ is already closed,  so that $I={\overline I}$.  A non-commutative version of this result is proved by Sinclair and Tullo
\cite[Lemma 1]{ST}, and this result is stated in \cite[Proposition 2.6.37]{D}, as follows.\s

 \begin{theorem} \label{1.3}
 Let $A$ be a Banach algebra, and let $I$ be a left ideal of $A$ such that ${\overline I}$ is finitely generated. Then $I$ is closed in $A$.\qed
 \end{theorem}\s
 
 In fact, the result is stated for more general algebras than Banach algebras.   Unfortunately, the proof of Proposition 2.6.37 in \cite{D} is not correct. 
Let $A$ be a Banach algebra, and take $n\in\N$. We define ${\mathfrak A} ={\mathbb M}_{\,n}(A)$, the algebra of $n\times n$ matrices over $A$, so that
  ${\mathfrak A}$ is also a Banach algebra with respect to the norm $\norm$, where 
$$
\lV a \rV =\max\{\lV a_{i1}\rV + \cdots + \lV a_{in}\rV  : i=1,\dots, n\}\quad (a=(a_{ij}) \in {\mathfrak A})\,;
$$
in the case where the algebra $A$ is unital, with identity $e_A$, the element $\iota=(\iota_{ij})$, where $\iota_{ij} =\delta_{ij}e_A\,\;(i,j=1,\dots,n)$ 
is the identity of ${\mathfrak A}$.  To prove Theorem \ref{1.3}, we may suppose that $A$ is unital. The proof in \cite{D}  refers to the `determinant' 
of elements in ${\mathfrak A}$; however, the determinant of such elements  is only defined in the special case where $A$ is commutative.
Nevertheless, the proof in \cite[Lemma 1]{ST} is correct; we sketch the details.

Let $I$ be a left ideal in $A$ with ${\overline I} = \langle a_1,\dots,a_n\rangle$, 
where $a_1,\dots,a_n\in \overline{I}$.  Then the open mapping theorem shows 
that there are $b_1, \dots, b_n\in I$  and $x =(x_{ij}) \in {\mathfrak A}$  with $\lV x \rV <1$ such that 
$$
a_i= b_i + x_{i1}a_1+\cdots + x_{in}a_n\quad (i=1,\dots, n)\,,
$$
and so $(b_1,\dots,b_n) = (\iota-x)(a_1, \dots, a_n)$ in $A^n$.   Since $\iota-x$ is invertible in ${\mathfrak A}$, it follows that
 $(a_1, \dots, a_n)= (\iota-x)^{-1}(b_1,\dots,b_n) \in I^n$, giving the result.
 
It follows that every left ideal in a Banach algebra $A$  is finitely generated whenever this is  true for each closed left ideal in $A$.

For various generalizations of versions of Theorem \ref{1.2} to certain topological algebras, see  \cite{CK,CKO,FT,Z1,Z2}. 
 
 The following  generalizations of Theorems \ref{1.2} and \ref{1.3} were  given by Boudi  in \cite[Proposition 1 and Theorem 3]{Boudi}.\s
 
  \begin{theorem} \label{1.3a}
Let $A$ be a Banach algebra, and let $I$ be a left ideal of $A$ such that ${\overline I}$ is countably generated.
 Then $I$ is closed in $A$.\qed
 \end{theorem}\s
 
   \begin{theorem} \label{1.3aa}
   Let $A$ be a Banach algebra. Suppose that every closed left ideal in $A$ is countably generated. Then $A$ is finite dimensional.\qed
\end{theorem}\s

 \begin{corollary} \label{1.3b}
Each closed, countably-generated left  ideal   in a Banach algebra  is finitely generated.
 \end{corollary}
 
 \begin{proof}  Let $I$ be a closed, countably-generated left ideal in a Banach algebra $A$, say $I  =\langle S\rangle$, where $S= \{a_n : n\in\N\}$.  For $n\in \N$, set 
 $J_n = \langle a_1, \dots, a_n\rangle$, so that $J_n \subset J_{n+1} \,\;(n\in \N)$ and $\bigcup\{J_n : n\in\N\} = I$.   By Baire's category theorem, 
there exists $n_0\in\N$ such that  ${\rm int\/}\overline{J}_{n_0} \neq \emptyset$.  But then $\overline{J}_{n_0} = I$, 
and so,  by Theorem \ref{1.3a},  $J_{n_0}$ is closed in $A$.  Thus $I = J_{n_0}$ is finitely generated.
 \end{proof}
\medskip

\section{The general case}

\subsection{The families ${\frak M}_\infty (A)$ and ${\frak N}_\infty (A)$} In this section, we shall consider algebras which are not necessarily  commutative. 

Let $A$ be  an algebra.  Then the families ${\frak I}_\infty (A)$ and  ${\frak U}_\infty (A)$, which were defined above, are each
 a partially ordered set with respect to inclusion. \s

\begin{theorem} \label{2.1}
Let $A$ be an algebra.  Then each member of the family ${\frak I}_\infty(A)$ is contained in a maximal element of the family. 
\end{theorem} 

\begin{proof}  Let $\mathcal C$ be a chain in the partially ordered set $({\frak I}_\infty(A), \subset )$, each member of which contains the 
specified member of the family,  and define $I= \bigcup\{J: J\in {\mathcal C}\}$, so that  $I$ is a left ideal in $A$.  

Assume towards a contradiction that $I$ is finitely generated, say $I=\langle a_1,\dots,a_n\rangle$, where $a_1,\dots,a_n\in I$. Then there 
exists $ J\in {\mathcal C}$ such that $a_1,\dots,a_n\in J$, and hence $J= \langle a_1,\dots,a_n\rangle$ is finitely generated, a contradiction. Thus 
$I\in {\frak I}_\infty (A)$.  Clearly,  $I$  is an upper bound for  $\mathcal C$.

It follows from Zorn's lemma that ${\frak I}_\infty(A)$ contains a maximal element, and that this maximal element contains the specified member of the family. 
\end{proof}\s

The following theorem   applies only to Banach algebras.\s

\begin{theorem} \label{2.1a}
Let $A$ be a    Banach algebra.  Then each member of the family ${\frak U}_\infty(A)$ 
is contained in a maximal element of the family. 
\end{theorem}

\begin{proof} As above,  consider a   chain $\mathcal C$   in the partially ordered set $({\frak U}_\infty(A), \subset)$, 
and define $I= \bigcup\{J: J\in {\mathcal C}\}$.  Then $I$ is a left ideal in $A$ and $\overline{I}$ is a closed left ideal in $A$.  

Assume towards a contradiction that $\overline{I}$ is {\CG}. By  Corollary \ref{1.3b}, $\overline{I}$  is finitely generated.  By Theorem \ref{1.3}, $I$
 is closed, and so $I$ is finitely generated.  As in Theorem \ref{2.1}, this is a contra\-diction. Thus  $\overline{I}\in {\frak U}_\infty(A)$. Clearly,
$\overline{I}$  is an upper bound for  $\mathcal C$ in ${\frak U}_\infty(A)$, and so Zorn's lemma again applies.
 \end{proof}\s

The sets of maximal elements in ${\frak I}_\infty(A)$ and ${\frak U}_\infty(A)$ are denoted by ${\frak M}_\infty(A)$ 
(for an algebra $A$) and ${\frak N}_\infty(A)$ (for a Banach algebra $A$), respectively.\s

 \begin{corollary} \label{2.1b}
 Let $A$ be an infinite-dimensional Banach algebra. Then the families  ${\frak M}_\infty(A)$ and 
${\frak N}_\infty(A)$ are    non-empty and equal, and each member of this family is closed. 
\end{corollary}

\begin{proof}  Clearly, both families are non-empty and each member of either of these families is closed because, by Theorem \ref{1.3a}, their closures belong to the respective families. 

Take $M \in {\frak M}_\infty(A)$. Then $M \in {\frak I}_\infty(A)$, and so there exists an element $N \in {\frak N}_\infty(A)$  with $M\subset N$. 
 The closed left ideal $N$  is not finitely generated, and so, by Corollary \ref{1.3b}, $N$ is not {\CG}. Thus
 $N \in {\frak I}_\infty(A)$, and so 
$M=N \in {\frak N}_\infty(A)$.

Take $N \in {\frak N}_\infty(A)$. Then $N \in  {\frak I}_\infty(A)$, and so there exists  an element  $M \in {\frak M}_\infty(A)$  with $N\subset M$.
By Corollary \ref{1.3b}, $M\in {\frak U}_\infty(A)$, and so $N=M \in {\frak M}_\infty(A)$.

We  have shown that ${\frak M}_\infty(A) = {\frak N}_\infty(A)$.
\end{proof}\s

We wish to study the following conjecture.\s

{\bf Conjecture}  {\it Let $A$ be a unital Banach algebra. Suppose that all maximal left ideals are finitely {\rm(}or even singly{\rm)} generated  in $A$. 
  Then $A$ is finite dimensional.}\s
  
We remark that this is a question about Banach algebras, not a purely algebraic  question.   For consider a large, infinite-dimensional field $F$ 
containing $\C$.  Then $F$ has only one proper ideal, namely $\{0\}$, and this is finitely generated, but $F$ is not finite-dimensional over $\C$. 
 However, by the Gel'fand--Mazur theorem \cite[Theorem 2.2.42]{D}, such a field  $F$  cannot be a Banach algebra.\s

The above conjecture  is  considered in \cite{DKKKL}  in the special case in which $A= {\B}(E)$, the Banach algebra of all bounded linear 
operators on a Banach space $E$, and it will be established there  for `many' Banach spaces.  The question is left open for the Banach algebra
 ${\B}(C(\I))$, where $C(\I)$ is the Banach space  of all continuous functions on the closed unit interval $\I$.

 We  make the following remark  about a special case of the conjecture. Let $A$ be a unital $C^*$-algebra. Then it can be shown rather easily that each
 finitely-generated left ideal in $A$ is singly generated by a self-adjoint projection,  and then that $A$ is finite dimensional whenever each maximal left ideal
 is finitely generated, and so our conjecture holds for the class of $C^*$-algebras.
 
   Let $A$ be a unital Banach algebra,  and take $M \in {\frak M}_\infty(A)$.  One might suspect that $M$ is necessarily a maximal left ideal in $A$; 
if this were true, then our conjecture would be immediately positively resolved. However this is not true.  Trivially, it is true in the special case 
where $\codim M =1$; the following example shows that it need not be true in the case where $\codim M =2$.\s

\begin{example}\label{3.12}
{\rm  We begin with  the unital, three-dimensional algebra $B$  which consists of the upper-triangular matrices in  ${\mathbb M}_{\,2}$.   
Thus we identify $$B =\C \/p \oplus \C\/ q\oplus \C\/ r\,,$$
 where 
  \begin{equation}\label{(R)}
p= \left(
\begin{array}{cc}
1&0\\
0&0
 \end{array}
\right)\,, \quad q= \left(
\begin{array}{cc}
0&0\\
0&1
 \end{array}
\right),\quad r = \left(
\begin{array}{cc}
0&1\\
0&0
 \end{array}\right).
\end{equation}
Hence the product in $B$ is specified by $$p^2=p,\; \;q^2=q,\;\;pq=qp=0,\;\;r^2=0,\;\;pr =rq=r,\;\; rp =qr = 0\,.
$$
The identity of $B$ is $e = p+q$; the radical of $B$ is $J(B) =\C r$. 
 
 We define $M= \C p$. Then $M$ is a left ideal in $B$  of codimension $2$.
 
 We further define $I= \C p \oplus \C r$ and $J = \C q \oplus \C r$. Then both $I$ and $J$ are left ideals in $B$ of codimension $1$,  both are maximal 
left ideals, and they are the only two maximal left ideals in $B$. Clearly, $M\subset I$, but $M\not\subset J$; further,  $I\cap J = \C\/r = J(B)$. 
 Since $M\subsetneq I$, the left ideal $M$  is not a maximal left ideal.

We define $$\lV \alpha p + \beta q + \gamma r\rV= \max\{\lv \alpha\rv,\lv \beta\rv,\lv \gamma\rv\}\quad (\alpha,\beta,\gamma \in \C)\,.
$$
Then $(B, \norm)$ is a Banach algebra with $\lV e \rV =1$.
 
 We now take $(E, \norm)$ to be an infinite-dimensional Banach space, and set $K=E\oplus E\oplus E\oplus E$; a generic element of $K$ is regarded as a $2\times 2$ matrix
 $$
{\bf x}=  (x_{i,j}) = \left(
\begin{array}{cc}
x_{1,1}&x_{1,2}\\
x_{2,1}&x_{2,2}
 \end{array}
\right)\,,
 $$
where $x_{1,1},x_{1,2},x_{2,1},x_{2,2}\in E$, and so $K={\mathbb M}_{\,2}(E)$ as a linear space.  The norm on $K$ is given by 
$$
\lV {\bf x}\rV = \lV  x_{1,1}\rV +\lV  x_{1,2}\rV  +\lV  x_{2,1}\rV + \lV  x_{2,2}\rV\quad ({\bf x} \in K)\,,
$$
so that $(K,\norm)$ is a Banach space.
 
 The left and right actions of $B$ on $K$ are   given by `matrix-multi\-plication  on the left and right', respectively; these actions are denoted by $\,\cdot\,$.
 Clearly, these are  associative actions, and $(K,\,\cdot\,,\norm)$ is a  unital Banach $B$-bimodule.
 
 We now define the linear space $A = B \oplus K$, with the norm given by 
$$
\lV (b, {\bf x}) \rV= \lV b \rV + \lV {\bf x}\rV\quad (b\in B,\,{\bf x}\in K) 
 $$
 and the product given by 
 $$
 (b, {\bf x})\,(c,{\bf y}) = (bc,\,b\,\cdot\,{\bf y} + {\bf x}\,\cdot\, c)\quad (b,c \in B,\,{\bf x}, {\bf y} \in K)\,.
 $$
Then $A$ is a unital Banach algebra, with identity $(e,0)$. We regard $B$ and $K$ as subspaces of $A$; clearly, $K$ is an ideal in $A$.

The space $K$ satisfies $K^2=\{0\}$, and so $K\subset J(A)$. Thus 
$$
J(A) = \C r \oplus K\,.
$$
  There are just two maximal left ideals  in $A$; they are $I+K$ and $J+K$, and $(I+K)\cap (J+K) = J(A)$.
The closed left ideal $M + K$ has codimension $2$ in $A$;  the only maximal left ideal that contains $M$ is $I+K$.

We  {\it  claim\/} that $I+K$ is a finitely-generated left ideal of $A$; indeed, we claim  that $I+K = Ap + Ar$. Since $p,r \in I$, we have $Ap+Ar\subset I+K$.
 Certainly, $p,r \in Ap + Ar$, and so $I\subset Ap+Ar$. Now take ${\bf x}\in K$. Then
$$
{\bf x} = \left(
\begin{array}{cc}
x_{1,1}&0\\
x_{2,1}&0
 \end{array}
\right)\left(
\begin{array}{cc}
1&0\\
0&0
 \end{array}
\right)
 +
 \left(
\begin{array}{cc}
x_{1,2}&0\\
x_{2,2}&0
 \end{array}
\right)\left(
\begin{array}{cc}
0&1\\
0&0
 \end{array}\right) \in Ap + Ar\,,
$$
and so $K\subset Ap+Ar$, giving the claim.

We also {\it  claim\/} that $M+K$ is not finitely generated.   Indeed, assume towards a contradiction that 
$$
K\subset Ap + \sum_{k=1}^n A{\bf x}^{(k)}\,,
$$
where ${\bf x}^{(1)},\dots, {\bf x}^{(n)}  \in K$.   Since $K^2=\{0\}$ and $Bp\cap K =\{0\}$, in fact
$$
K\subset Kp + \sum_{k=1}^n B{\bf x}^{(k)}\,.
$$
In particular,  for each $x \in E$,  there exits $\alpha_k, \beta_k, \gamma_k \in \C$ for $k=1,\dots,n$ such that
$$
\left(
\begin{array}{cc}
0&0\\
0&x
 \end{array}
\right) \in Kp + \sum_{k=1}^n  \left(
\begin{array}{cc}
\alpha_k&\gamma_k\\
0&\beta_k
 \end{array}
\right){\bf x}^{(k)}\,,
$$
and so
$$x = \sum_{k=1}^n \beta_k x_{2,2}^{(k)}\,.$$
This shows that $E$ is spanned by $\{x_{2,2}^{(1)},\dots,x_{2,2}^{(n)}\}$, a contradiction of the fact that $E$ is infinite dimensional.  This gives the claim.

Thus $M +K\in {\mathfrak M}_\infty(A)$, but $M+K$ is not a maximal left ideal of $A$.

We also note that the maximal left ideal $J+K$ of $A$ is not finitely  generated, and so our example is not a counter-example to our conjecture.

Indeed, assume towards a contradiction that 
$$
K \subset Aq + Ar + \sum_{k=1}^n A{\bf x}^{(k)}\,,
$$
where ${\bf x}^{(1)},\dots, {\bf x}^{(n)}  \in K$.  By considering the element 
$$
\left(
\begin{array}{cc}
x&0\\
0&0
 \end{array}
\right)
$$
of $K$, we see that $E$ is spanned by $\{x_{1,1}^{(1)},\dots,x_{1,1}^{(n)}, x_{1,2}^{(1)},\dots,x_{1,2}^{(n)}\}$, again a contradiction of the
 fact that $E$ is infinite dimensional.}\qed
 \end{example}

Let $A$ be a unital Banach algebra. For a positive resolution of the above conjecture, we must prove that, whenever there exists an  element $M \in {\frak M}_\infty(A)$, 
 $A$ contains a maximal left ideal that is not finitely generated (but this ideal need not contain $M$).  There is an algebraic argument that 
does establish this in the special case where   $M$ has finite codimension in $A$. \s

\begin{lemma}\label{1.6}
Let $I$ be a left ideal in ${\mathbb M}_{\,n}$, where $n\in\N$.   Then:\s

{\rm (i)} $\dim I =kn$ for some $k\in \{0,\dots, n\}$;\s

{\rm (ii)}  $I$ is a maximal left ideal if and only if $\codim I= n$ in ${\mathbb M}_{\,n}$;\s  

{\rm (iii)}    in the case where $I$ is proper, there exists $x \in\C^{\,n}$ with $x\neq 0$ such that $ax =0\,\;(a\in I)$.
\end{lemma}

\begin{proof}   By \cite[Exercise 3, p.~173]{J}, there is a bijective map
$$
F\mapsto \{a \in \M_{\,n}: ax=0\;\,(x\in F)\}
$$
from the family of linear subspaces of $\C^{\,n}$ onto the family of left ideals of $\M_{\,n}$.  The result follows easily from this.
\end{proof}\s

\begin{lemma}\label{1.7}
Let $A$ be a unital algebra with an ideal $L$ such that $A/L = {\mathbb M}_{\,n}$ for some $n \in \N$.
 Suppose that each maximal left ideal in $A$ that contains $L$ is finitely generated in $A$. Then $L$ is finitely generated  as a left ideal in $A$.
\end{lemma}

\begin{proof}  We identify $A$ with $\M_{\,n}\oplus L$  as a linear space, and  denote the product of matrices in ${\mathbb M}_{\,n}$ by $\,\cdot\,$.

Take  $j\in\N_n$.  Define $P_j = \iota_n-E_{j,j}$, where we recall that $(E_{i,j})$ is the set of matrix units of $\M_{\,n}$, 
 and set $ M_j =\langle P_j\rangle\subset \M_{\,n}$, so that $M_j$ is the space of matrices with zeros
 in the $j^{\rm th}$ column. By Lemma \ref{1.6}(ii), $M_j$ is a maximal left ideal in ${\mathbb M}_{\,n}$,   and so $M_j+L$ is a  maximal left ideal in $A$. 
By hypothesis, $M_j+L$  is finitely generated in $A$, and so there is a finite subset $S_j$ in $L$ such that  $L= AP_j +\langle S_j\rangle$.

Set $S = S_1\cup \cdots\cup S_n$, a finite subset of $L$. Then  
$$
L\subset  AP_j +\langle S\rangle\quad(j\in\N_n)\,.
$$ 
By enlarging $S$, if necessary, we may suppose
 that $S$ contains the difference between the products in $A$ and $\M_{\,n}$ of any two  matrix units   $E_{i,j}$ in $\M_{\,n}$, and so we have 
 $$
L\subset   {\mathbb M}_{\,n} \,\cdot\,P_j + LP_j +\langle S\rangle\quad(j\in\N_n)\,.
$$ 
Since $LP_j+ \langle S\rangle\subset L$  (where we note that $L$ is a right ideal) and since $(\M_{\,n}\,\cdot\,P_j) \cap L =\{0\}$, it follows that 
 $$
L =   LP_j +\langle S\rangle\quad(j\in\N_n)\,.
$$ 

Set $T = S\cup SP_1 \cup \cdots \cup SP_n$, a finite subset of $L$. Then we now have 
\begin{eqnarray*}
L &= & LP_1 + \langle S\rangle = (LP_2+ \langle S\rangle)P_1 + \langle S\rangle = LP_2P_1 + \langle S\cup SP_1\rangle\\
&=&  \cdots \;=\\
&=&  LP_n\cdots P_1 + \langle T\rangle = \langle T\rangle
\end{eqnarray*}
because the product of $P_n, \dots, P_1$ in  ${\mathbb M}_{\,n}$ is zero, and so the product in $A$ is an element of $T$.

This proves the result.
\end{proof}\s

\begin{lemma}\label{1.8}
Let $A$ be a unital algebra, and take ideals $I$, $J$, and $L$ in $A$  such  that $L = IJ +JI$.  Suppose that $I$ and $J$ are finitely generated as
 left ideals in $A$. Then $L$ is a finitely-generated   left ideal in $A$.
\end{lemma}

\begin{proof} 
 Suppose that $I= Ax_1 + \cdots +Ax_m$ and $J= Ay_1+ \cdots +Ay_n$, where $x_1,\dots,x_m \in I$ and $y_1,\dots, y_n\in J$. 
 Then  $$x_iy_j \in IJ\subset L\quad(i\in\N_m, j\in\N_n)\,.
 $$
Take $x\in I$ and $y \in J$.  Then $y=\sum_{j=1}^n a_jy_j$ for some $a_1,\dots, a_n \in A$.   For each $j\in \N_n$, we have  $xa_j \in I$, and so 
$$
xa_j = \sum_{i=1}^m b_{i,j}x_i 
$$
for some $b_{1,j}, \dots, b_{m,j} \in A$. Hence
$$
xy = \sum_{i=1}^m\sum_{j=1}^n b_{i,j}x_iy_j \in  \langle S_1\rangle\,,
$$
where $S_1$ is the finite set $\{x_iy_j : i\in \N_m,j\in\N_n\}$.  It follows that $IJ\subset \langle S_1\rangle$.
Similarly, we see that $JI\subset \langle S_2\rangle$, where $S_2$ is the finite set $\{y_jx_i : i\in \N_m,j\in\N_n\}$.

It follows that $L$ is generated by the finite set $S_1\cup S_2$.
\end{proof}\s
 
 \begin{lemma}\label{1.9}
 Let $A$ be a unital algebra, and let $L$ be an ideal in $A$ of finite codimension such that $A/L$ is semisimple.  
 Suppose that each maximal left ideal in $A$ that contains $L$ is finitely generated in $A$. Then $L$ is finitely generated  as a left ideal in $A$.
 \end{lemma}
 
 \begin{proof}  By Wedderburn's theorem   \cite[Theorem 1.5.9]{D}, we have $A/L = {\mathbb M}_{\,n_1}\oplus \cdots \oplus {\mathbb M}_{\,n_k}$ 
for some $n_1,\dots,n_k \in \N$.  For $j\in \N_k$, set 
 $$
 I_j = \bigoplus \{{\mathbb M}_{\,n_i} : i\in\N_k,\, i\neq j\}+ L\,,
 $$
 so that $I_j$ is an ideal in $A$ with $A/I_j = \M_{\,n_j}$. By Lemma \ref{1.7}, each $I_j$  is a finitely-generated left ideal in $A$.  

Take $j\in \N_k$.  The identity matrix in ${\mathbb M}_{\,n_j}$  is now denoted by $p_j$, and so the identity of $A$ has the form 
$p_1+ \cdots+ p_k+ x_0$ for some $x_0\in L$.

Take $j_1,j_2 \in\N_k$ with $j_1\neq j_2$.  Clearly $I_{j_1}I_{j_2}+ I_{j_2}I_{j_1} \subset  I_{j_1}\cap I_{j_2}$. Now take 
$x \in  I_{j_1}\cap I_{j_2}$. We have 
$$
x = (p_1+ \cdots+ p_k+ x_0)x \in I_{j_1}I_{j_2}+ I_{j_2}I_{j_1}
$$
because each $p_i$ for $i\in\N_n$ belongs to either $I_{j_1}$ or to $I_{j_2}$ (or to both)  and  $x_0 \in I_{j_1}\cap I_{j_2}$.  Thus 
$I_{j_1}\cap I_{j_2}\subset I_{j_1}I_{j_2}+ I_{j_2}I_{j_1}$. We have shown that  $I_{j_1}I_{j_2}+ I_{j_2}I_{j_1} =  I_{j_1}\cap I_{j_2}$.  By Lemma \ref{1.8},
$I_{j_1}\cap I_{j_2}$ is a finitely-generated   left ideal in $A$.

By repeating  this argument finitely many times, we see that the ideal $I_{1}\cap \cdots \cap I_{n}$ is a finitely generated  left ideal in $A$.
 However, $I_{1}\cap \cdots \cap I_{n} = L$, and so $L$ is finitely generated  as a left ideal in $A$.
 \end{proof}\s
 
  \begin{lemma}\label{1.10}
 Let $A$ be a unital algebra, and let $K$ be an ideal in $A$ of finite codimension.
 Suppose that each maximal left ideal in $A$ that contains $K$ is finitely generated in $A$. Then $K$ is finitely generated  as a left ideal in $A$.
 \end{lemma}
 
  \begin{proof}  By Wedderburn's principal theorem   \cite[Corollary 1.5.19]{D}, $A/K = B\oplus R$ for a subalgebra $B$ 
which is a direct sum of full matrix algebras and an ideal $R$ which is the radical of $A/K$. We identify $A$ with $B\oplus R\oplus K$ as a linear space. 
The  ideal $R$ is nilpotent in $A/K$, say  $R^n =\{0\}$ in $A/K$, and so $R^n\subset K$ in $A$.

Set $L = R + K$. Then $L$ is an ideal of $A$ such that $A/L$ is semisimple, and so, by Lemma \ref{1.9}, $L$ 
is finitely generated  as a left ideal in $A$.  Thus there exist $r_1,\dots, r_m \in R$ and a finite subset $S$ of $K$ such that 
 $$
 K \subset Ar_1 + \cdots + Ar_m + \langle S \rangle\,.
 $$
 
By enlarging $S$,  if necessary, we may suppose $Br_1, \dots, Br_m \subset S$  and that $S$ contains any product $r_{i_1}\cdots r_{i_k}$ for any 
$i_1,\dots, i_k \in\N_m$ and $k\in\N_n$; the enlarged set $S$ is still finite. Thus we see that 
$$
K\subset \sum_{i=1}^n K r_i + \langle S \rangle\,.
$$ 
But now 
$$
K\subset  \sum_{i_1,i_2=1}^m K r_{i_1}r_{i_2} + \langle S \rangle \subset \cdots \subset 
\sum_{i_1,\dots,i_n=1}^m K r_{i_1}\cdots r_{i_n} + \langle S \rangle = \langle S \rangle\,.
$$
Thus $K$ is finitely generated  as a left ideal in $A$.
\end{proof}\s
  
 \begin{theorem}   Let $A$ be a unital, infinite-dimensional   algebra, and suppose that some element of $\,{\mathfrak M}_\infty(A)$ has finite codimension   in  $A$.
 Then one of the maximal left ideals in $A$ is not finitely generated.
  \end{theorem}
  
 \begin{proof} We suppose that ${\mathfrak M}_\infty(A)$ is such that  $\codim M =n$, where $n\in\N$, and we identify $A/M$ with $\C^{\,n}$ as a linear space.
 The identity of $A$ is denoted by $e$. 
 
For each $a \in A$, set $T_a(b + M) = ab +M\,\;(b\in A)$. Then 
$$
\theta : a\mapsto T_a\,,\,\;A \to {\mathbb M}_{\,n}\,,
$$
 is a unital homomorphism,  and its image $\theta(A)$ is a unital subalgebra of ${\mathbb M}_{\,n}$.  Define 
 $$
 K =\ker \theta =\{ a \in A : aA \subset M\}\,,
 $$
 so that $K$ is an ideal in $A$ of codimension at most $n^2$.  For each 
$a \in A$, we have $a=ae \in M$, and so $K\subset M$.  Clearly $M/K$ is a left ideal in $A/K$.

Assume towards a contradiction that each maximal left ideal in $A$ is finitely generated.  By Lemma \ref{1.10}, $K$ is finitely generated  as a left ideal
in $A$, and so $M$ is also finitely generated  as a left ideal in $A$ because $M/K$ is a finite-dimensional space, a contradiction of the hypothesis.

Thus $A$ contains a maximal left ideal that is not finitely generated.
\end{proof}\s

We now introduce a slightly complicated algebraic condition.\s

\begin{theorem} \label{2.3}
Let $A$ be an algebra.  Suppose that $I$ is a  left ideal in $A$ with the property that
there exist $a,b \in A\setminus I$ such that $ba \in I $ and $Ia\subset I$.  Then $I$ does not belong to  either ${\frak M}_\infty(A)$ or ${\frak U}_\infty(A)$.
\end{theorem} 

\begin{proof}  Assume towards a contradiction that $I\in {\frak M}_\infty(A)$.  We note that the   every left ideal $J$ that properly contains 
$I$ is finitely generated.

Take $a$ and $b$ as specified.

Consider the left ideal $A^{\sharp}a+ I$ of $A$. Since $a\notin I$, the left ideal  $A^{\sharp}a+ I$ properly contains $I$, and so 
is finitely generated, say 
$$
 A^{\sharp}a+ I=  \langle b_1,\dots,b_m\rangle\,,
$$ 
where $b_1,\dots,b_m \in A^{\sharp}a+ I$. Each element $b_i$ has the form $a_ia +u_i$,
 where $a_i \in A^{\sharp}$ and $u_i\in I$, and so  $A^{\sharp}a+ I = \langle a, u_1,\dots,u_m\rangle$.

Define $J= \langle   u_1,\dots,u_m\rangle\subset I$, so that $A^{\sharp}a + I = A^{\sharp}a + J$, 
and define  $K = \{c\in A : ca \in I\}$.   Then $K$ is a left ideal in $A$, and $I \subset K$ because $Ia\subset I$.
Further, we claim that $I\subset Ka +J$. For take $x \in I$. Since $I\subset A^{\sharp}a +J$, 
there exist $c \in A^{\sharp}$ and $j\in J$ with $x =ca+j$. Then $ca= x-j\in I$, and so $c\in K$, giving the claim.
Since $Ka+ J \subset I$, it  follows that $Ka+ J =I$.

Since $b\in K\setminus I$, we have $I \subsetneq K$, and so $K$ is finitely generated, say $K= \langle c_1,\dots,c_n\rangle$, where $c_1,\dots,c_n \in A$.
But now $$I= \langle c_1a,\dots,c_na, u_1,\dots,u_m\rangle\,,
$$
 and so $I$ is finitely generated, a contradiction.  Thus $I \not\in {\frak M}_\infty(A)$.
\end{proof} \s

\begin{theorem} \label{3.1}
Let $A$ be a Banach algebra.  Suppose that $I$ is an ideal in $A$ and that 
$I\in{\frak M}_\infty(A)$. Then either $I$ is a maximal modular  ideal of co\-dimension $1$ in $A$ or $I=A$.
\end{theorem}

\begin{proof}  We consider the case where  $I\neq A$.   By Corollary \ref{2.1b}, $I$ is closed.

Since $I$ is a closed  ideal in $A$, the space $A/I$ is a Banach algebra.  Each non-zero, left ideal in $A/I$ has the form $J/I$, where $J$ is a left 
 ideal in $A$ with  $J\supsetneq I$. Since $J$ is finitely generated  in $A$ as a left ideal, the ideal $J/I$ is a finitely-generated left ideal 
 in $A/I$,  and so  ${\frak I}_\infty(A/I)$ is  empty. By Theorem \ref{1.2},  $A/I$ is a finite-dimensional algebra.
 
Suppose that $A/I$ is semisimple.  Then $A/I$ is a finite direct sum of matrix algebras.  Assume towards a contradiction that  $\dim A/I\geq 2$. Then there 
 are idempotents $p$ and $q$ in $A/I$ with $pq=0$, and so there are $a,b \in A\setminus  I$ with $ab \in I$. Since $Ia \subset I$,
 it is immediate from Theorem \ref{2.3}  that $I\notin{\frak M}_\infty(A)$, a contradiction.  
Thus $\dim A/I=1$, and so $I$ is a maximal modular  ideal of co\-dimension $1$ in $A$.

 Assume towards a contradiction that $A/I$ is not semisimple. Since the radical of $A/I$ is finite dimensional, there exists $a\in A\setminus I$ with $a^2 \in I$, again  
a contradiction of Theorem \ref{2.3}.  

 Hence $I$ is a maximal modular ideal of codimension $1$ in $A$.\end{proof}
\medskip

\section{The commutative case}
\subsection{Finitely-generated maximal ideals} We shall now consider commutative Banach algebras.\s

\begin{theorem} \label{3.2}
Let $A$ be a commutative, unital Banach algebra.  Suppose that each   maximal   ideal in $A$ is {\CG}.
Then $A$ is finite dimensional. 
\end{theorem}

\begin{proof} Since maximal ideals in $A$ are closed, each maximal ideal in  $A$ is finitely generated by Corollary  \ref{1.3b}.

Assume  towards a contradiction that ${\frak M}_\infty(A)\neq \emptyset$,  and then take $I\in{\frak M}_\infty(A)$.   Since $A$ is unital, with 
 identity $e_A$, say,  $A$ itself is singly generated by $e_A$, and so $I \neq A$. By  Theorem \ref{3.1}, $I$ is a maximal ideal in $A$. Since each   maximal  ideal is 
finitely generated, we have a contradiction, and so  ${\frak M}_\infty(A)= \emptyset$. By Corollary \ref{2.1b}, $A$ is finite dimensional.
\end{proof}\s

The above theorem  requires that $A$ be unital, and this condition cannot be dispensed with. For let $A$ be the commutative, radical algebra $L^1([0,1])$, with convolution multiplication, as 
in \cite[Definition 4.7.38]{D}; thus $A$ is the {\it Volterra algebra\/}.  Then  $A$ has no maximal ideals, but $A$ is infinite dimensional.

 Let $A$ be a commutative Banach algebra. It can happen that the sets ${\frak M}_\infty(A)$  has just one element, 
 whilst the maximal ideal space of $A$ is infinite. For example, let $A$ be the algebra $c$, so that $c$ is a   commutative, unital 
Banach algebra. A  maximal ideal  of $c$ of the form 
$$
M_k=\{(x_i) \in c : x_k=0\}\,,
$$
 where $k\in \N$, is generated by the sequence $(a_i)$, where 
$a_k=0$ and $a_i=1\,\;(i\neq k)$.   The only other maximal idea of $c$ is 
$$
c_{\,0} = \{(x_i) \in c : \lim_{i\to \infty} x_i=0\}\,,
$$
and this ideal is not finitely generated, and hence not {\CG}.\medskip

\subsection{The {\v S}ilov boundary} 
 We first recall some standard notation;  see \cite[Chapter 4]{D}, for example.   

Let $K$ be a non-empty, compact space. We write $C(K)$ for the  space of all continuous functions on  $K$ with the pointwise operations,  and  set
 $$
 \lv f \rv_K =\sup\{\lv f(x)\rv : x\in K\}\quad (f\in C(K))\,,
 $$
so that $(C(K), \lv \,\cdot\,\rv_K)$ is a commutative, unital  Banach algebra. A unital subalgebra $A$ of $C(K)$ that separates the points of $K$ and is a Banach
 algebra for some norm $\norm$ is a {\it Banach function algebra\/} on $K$; $A$ is a {\it uniform algebra\/} if it is closed in $(C(K), \lv \,\cdot\,\rv_K)$.  
 
A Banach function algebra $A$ on $K$ is {\it natural\/} if every character on $A$ has the form $\varepsilon_x: f\mapsto f(x)$ for some $x \in K$ or, 
equivalently, each maximal ideal of $A$ has the form $M_x=\{f\in A : f(x)=0\}$ for some $x \in K$.   

Let $A$ be a   commutative, unital Banach algebra, and let $\Phi_A$ be the character space of $A$ as in \cite{D},
 so that $\Phi_A$ is identified with the  maximal ideal space of $A$; the maximal ideal corresponding to a character $\varphi\in \Phi_A$ is
$M_\varphi = \ker \varphi$. It is standard that $\Phi_A$ is a  compact space in the relative weak-$*$ topology $\sigma(A',A)$
 and that the Gel'fand transformation  maps $A$ onto a  Banach function algebra $\widehat{A}$ which is  natural on $\Phi_A$.
 
  Let $A$ be a natural Banach function algebra on a compact space $K$. A closed subset $F$ of $K$ is a {\it peak set\/} if there exists $f \in A$ such that
 $\lv f(x)\rv =1\,\;(x\in F)$ and  $\lv f(y)\rv <1\,\;(y\in K\setminus F)$; in this case, $f$ {\it peaks\/} on $F$; a point $x\in K$ is a  {\it peak point\/} 
if $\{x\}$ is a peak set, and a {\it strong boundary point\/} if $\{x\}$  is an intersection of peak sets; the set of strong boundary points of $A$ is denoted 
by $S_0(A)$.  A countable intersection of peak sets  is always a peak set, and so, in the case where $K$ is metrizable, $S_0(A)$ is the set of peak points of $A$. 
(However even  a uniform algebra may have strong boundary points which are not peak points.)  A subset $S$ of $K$ is a {\it boundary\/} if $S\cap F\neq \emptyset$ 
for each peak set $F$ of $A$; the intersection of all the closed boundaries of $A$ is a closed boundary, called the    {\it {\v S}ilov boundary\/} and 
denoted by   $\Gamma(A)$ \cite[Definition 4.3.1(iv)]{D}; for $x\in K$, we have $x \in   \Gamma(A)$ if and only if, for each open neighbourhood $U$
 of $x$ in $K$, there exists $f \in A$ with $\lv f\rv_K > \lv f\rv_{K\setminus U}$  \cite[Theorem 4.3.5]{D}.
 
In the case where $A$ is a natural
 uniform algebra on $K$, the  {\it Choquet boundary\/}  of $A$ is defined in \cite[Definition 4.3.3]{D}; 
 by  \cite[Theorem 4.3.5]{D}, it coincides with the  set $S_0(A)$; by \cite[Proposition 4.3.4]{D}, $S_0(A)$
 is a boundary for $A$ and, by  \cite[Corollary 4.3.7(i)]{D}, $S_0(A)$ 
is a dense subset  of $\Gamma(A)$.

 Let $A$ be a commutative, unital Banach algebra. Then we define $\Gamma(A)$ to be $\Gamma(\widehat{A})$.

The following two results are essentially  Corollary 1.7 of \cite{FT}, where an analogous result for more general topological algebras is proved.
 \s

\begin{theorem} \label{3.3}
Let $A$ be a commutative, unital  Banach algebra.  Suppose that $\varphi \in  \Gamma(A)$ and  the  ideal $M_\varphi$ is 
{\CG}.  Then $\varphi$ is isolated in $\Phi_A$.
\end{theorem}

\begin{proof} Since $M_\varphi$ is a closed ideal in $A$, it follows from  Corollary \ref{1.3b} that $M_\varphi$ is  finitely generated. 

By  a theorem of Gleason (e.g., see \cite[Theorem 15.2]{Stout}),  there is an open neighbourhood $U$ of $\varphi$ in  $\Phi_A$ 
 and a homeomorphism $\gamma$ from $U$ onto an analytic variety   $V$  in a polydisc $\Delta$ in $\C^n$ for some $n\in \N$ such that, for each $a\in A$, 
there is a holomorphic function $F$ on $\Delta$ such  that $\widehat{a} = F\,\circ\,\gamma$ on $U$.

We have $\partial U \subset K\setminus L$, where $\partial U$ denotes the frontier of $U$.  Assume that $\partial U \neq \emptyset$. Then
 there exists $z\in V$ such that $\lv F(z)\rv > \lv F \rv_{\partial V}$  for a holomorphic function $F$ on $\Delta$, a contradiction of the
 maximum principle for holomorphic functions on varieties \cite[III, Theorem 16]{GRo}. Thus $\partial U =\emptyset$ and $U$ is compact.  Hence $V$  is compact. 
But there are no compact, infinite varieties in $\C^{\,n}$, and so $V$  and $U$ are finite. 
 Thus $\varphi$ is isolated in $\Phi_A$. 
\end{proof}\s

\begin{theorem} \label{3.3c}
Let $A$ be a commutative, unital  Banach algebra.  Suppose that $M_\varphi$ is {\CG}
 for each $\varphi \in  \Gamma(A)$.  Then $A$ is finite dimensional.
\end{theorem}

\begin{proof}  Since $\Gamma (A)$ is compact and each point of  $\Gamma(A)$ is isolated in $\Phi_A$, the set   $\Gamma(A)$ is finite, and hence $\Phi_A = \Gamma(A)$.
By Theorem \ref{3.2}, $A$ is  finite dimensional.  \end{proof}\s

One might suspect that ${\frak M}_\infty(A)\subset \Gamma(A)$ for a {\CBA} $A$,
but the following  examples show that this is not necessarily the case.\s  

\begin{example}\label{3.4}
{\rm  Let $A$ be an infinite-dimensional, commutative, unital  Banach algebra that is local, so that the unique maximal ideal
 in $A$ is $J(A)$. Then $J(A)$  is not finitely generated.   

However an infinite-dimensional radical of a  commutative, unital  Banach algebra can be singly generated. 
For example, let $B$ be  any infinite-dimensional, commutative, unital  Banach algebra, with identity $e$, and let $A$ be a unital subalgebra of $B$. 
Set ${\mathfrak A}  = A\oplus B$, with  $$
\lV (a,b)\rV = \lV a \rV + \lV b \rV\quad(a  \in  A,\, b\in B)
$$
 and product defined by 
$$
(a_1,b_1)(a_2,b_2)= (a_1a_2, a_1b_2+ a_2b_1)\quad (a_1 ,a_2 \in A,\,b_1,b_2 \in B)\,.
$$
Then ${\mathfrak A}$ is a commutative Banach algebra with identity  $(e,0)$. In the case where $A=B$ is semisimple,  $J({\mathfrak A}) =\{0\}\oplus B$, 
which is infinite-dimensional; here the radical $J({\mathfrak A}) $ is generated by $(0,e)$ because $(0,b)= (b,0)(0,e)\,\;(b\in B)$.

Now suppose  that $B=C(\overline{\D})$, the uniform algebra  of all continuous functions on $\overline{\D}$, and that $A= A(\overline{\D})$, 
the disc algebra, consisting of the functions in $C(\overline{\D})$ that are analytic on $\D$.  Then $\Phi_{{\mathfrak A}} = \overline{\D}$ and $\Gamma({\mathfrak A}) =\T$.
We recall that $$\{f\in A: f(0) =0\}= ZA\,,
$$
 where $Z$ is the coordinate functional.  Set
$$
M = \{(f,g) \in  {\mathfrak A} : f(0) =0\}\,.
$$
Then $M$ is a  maximal ideal of ${\mathfrak A}$, and it 
corresponds to a character on ${\mathfrak A}$ which is not in $\Gamma({\mathfrak A})$.

We {\it claim\/} that the ideal $M$ is not finitely generated in ${\mathfrak A}$. Indeed, assume towards a contradiction that 
$M$ is   finitely generated.  Then we can suppose that the generators are $(Zf_1,g_1), \dots ,(Zf_k,g_k)$, where $f_1,\dots,f_k \in A$ and 
$g_1,\dots,g_k \in B$. Thus, for each $g \in B$, there exist  $r_1,\dots,r_k \in A$ and $s_1,\dots,s_k \in B$ such that
\begin{equation}\label{(3.4)}
 (0,g) = \sum_{i=1}^k (r_i,s_i)(Zf_i, g_i)\,.
 \end{equation}
Set  $F= {\rm lin\/}\{g_1,\dots,g_k\}$, a  finite-dimensional subspace of $B$. Since  $r_i-r_i(0)1 \in ZA$, it follows 
from (\ref{(3.4)}) that $g \in \sum_{i=1}^k r_i(0)g_i +ZB$. Thus $B = F + ZB$.  However it is not true that $B/ZB$  is a finite-dimensional space; 
for example, the set  $\{\lv Z\rv^{1/n}+ZB: n\in\N$ is linearly independent  in  $B/ZB$.  This is the required contradiction.

Hence $M \in {\frak M}_\infty({\mathfrak A})$,   but $M \not\in \Gamma({\mathfrak A})$.\qed}
\end{example}\s

We now  present a natural uniform algebra $A$ for which  we have ${\frak M}_\infty(A)\not\subset \Gamma(A)$.  \s

\begin{example}\label{3.5}
 {\rm  Let $A$ be a unital, {\CBA}, and take $\varphi \in \Phi_A$; set $M=M_\varphi$.

Suppose that  $M = \langle f_1,\dots,f_n\rangle$.  Then $\dim (M/M^2 )\leq n$. 
Indeed, for each $f \in M$, there exist $g_1,\dots,g_n$  with $f=\sum_{j=1}^ng_jf_j$; for $j=1,\dots, n$, write $g_j = g_j(\varphi)1 + h_j$, 
 where $h_j \in M$. Then  
$$
f\in \sum_{j=1}^ng_j(\varphi)f_j +M^2\,,
$$
 and so $M/M^2 ={\rm lin\,}\{f_1+M^2, \dots, f_n+M^2\}$.  Thus the space
of point derivations at $\varphi$ is finite-dimensional.

Now let $X= \prod_{n=1}^\infty \Delta_n$, where each $\Delta_n$ is the closed unit disc, and let $A$ be the  tensor product of countably many copies of the disc algebra, 
as in \cite[Theorem 14.3]{Stout}. Then $A$ is semisimple, the character space of $A$ is $X$, and the {\v S}ilov boundary is $\Gamma (A)= \prod_{n=1}^\infty \T_n$,
 where each  $\T_n$ is the unit circle. The character corresponding to evaluation at the point $0=(0,0,\dots)$ is not in $\Gamma (A)$. 
 Let $M$ be the corresponding maximal ideal.  Then each $f\in M$ is an analytic  function in each of the coordinate functionals $Z_1, Z_2, \dots  $,
 and each linear functional $d_n: f\mapsto (\partial f/\partial z_n)(0)$ is a (continuous) point derivation at $0$. Since these linear functionals are
 linearly independent, it is not true that $\dim (M/M^2)$ is finite,  and so $M$ is not finitely generated.\qed}
\end{example}
\medskip

\subsection{Strong boundary points}
Let $A$ be a be natural Banach function  algebra, and now suppose that  the closed ideal $M_\varphi$ is  {\CG} for 
each   $\varphi\in S_0(A)$, rather than for each  $\varphi \in \Gamma(A)$.  Is it still true that  $A$ must be  finite dimensional?   
First we claim that this is true when $A$ is a uniform algebra  (and when the number of isolated points in $\Phi_A$ is countable).

We write $\delta_\varphi$ for the
 characteristic function of the singleton set  $\{\varphi\}$ when $\varphi \in \Phi_A$. By {\v S}ilov's idempotent theorem \cite[Theorem 2.4.33]{D}, 
 $\delta_\varphi\in A$ whenever $\varphi$ is isolated in $\Phi_A$.\s
 
 \begin{theorem} \label{3.3a}
 Let $A$ be a  uniform algebra on $\Phi_A$.  Suppose that $S_0(A)$ is countable and that the maximal ideal   $M_\varphi$ is 
{\CG} for each $\varphi \in  S_0(A)$.  Then $A$ is finite dimensional.
 \end{theorem}

\begin{proof} We set $S_0(A)= \{\varphi_n : n\in\N\}$.
By  Theorem \ref{3.3}, each point $\varphi \in S_0(A)$ is an isolated point of $\Phi_A$, and so $S_0(A) $ 
is open in $\Phi_A$. 

Assume towards a contradiction that we have $S_0(A)\neq \Phi_A$, and set $L=\Phi_A\setminus S_0(A)$, a non-empty, compact subset of $\Phi_A$.  
 Clearly each isolated point is  a peak point of $\Phi_A$, and so in $S_0(A)$.

Consider the function 
$$
f:= 1 -\sum_{n=1}^\infty \frac{1}{n}\delta_{\varphi_n}\,,
$$
 so that $f \in A$.   At each $\varphi \in S_0(A)$, we have  $\lv f(\varphi)\rv < 1$, and, for each $\varphi \in L$, we have $f(\varphi)=1$. Hence $L$ is a
 peak  set for $A$.    Since  $S_0(A)$ is a boundary  for $A$,    $S_0(A)\cap L\neq \emptyset$, a contradiction. 
 Thus  $S_0(A)=\Phi_A$. 

It follows from Theorem \ref{3.3c} that $A$ is finite  dimensional.\end{proof}\s

We do  not know if the above  theorem holds in the case where  $S_0(A)$ is not necessarily countable.\s

Let $A$ be a uniform algebra, and let $\varphi \in S_0(A)$.  Then it is easy to see that the following are equivalent without resource to Gleason's theorem,
 which was used  in the proof of Theorem  \ref{3.3}: \s

(a) $M_\varphi$ is singly generated; \s

(b)  $M_\varphi$  is {\CG}; \s

(c) $\varphi$ is isolated in $\Phi_A$. \s

\noindent It is sufficient to prove that (b) $\Rightarrow$ (c).  Thus, suppose that   $M_\varphi$  is {\CG}, and assume towards a contradiction that $\varphi$ is not
isolated in $\Phi_A$.  Then it is easy to see that, given any countable  set $S$ in $M_\varphi$,  there exists $p \in C(\Phi_A)$ 
 such that  such that $\lv p \rv_{\Phi_A} < 1$ and  $\lim_{\psi\to \varphi}p(\psi)/g(\psi) = \infty$ for each $g\in \langle S \rangle$.
By \cite[Theorem 20.12]{Stout}, there exists  $f\in A$ with $f(\varphi) =1$ and $\lv  f(\psi)\rv < 1- p(\psi)$ for each $\psi \in \Phi_A$ with $\psi \neq \varphi$.  
But now $1-f \in M_\varphi$, but $1-f\not\in  \langle S \rangle$.

Thus we have a direct proof of Theorem \ref{3.3a} avoiding Gleason's theorem.
 
 Let $\omega_1$ be the first uncountable   ordinal, set $A = C([0,\omega_1])$, and consider the maximal ideal $M_{\omega_1}$. 
Then the above remark shows that  $M_{\omega_1}$ is not {\CG}. This is related to \cite{Z3,Z4}.

We   show finally that the above theorem does not hold  if we replace   `uniform algebra' by `Banach function algebra'.\s 

 \begin{example} \label{3.3b}
{\rm   Since our example is rather long, we divide the construction into a number of steps.\s

(1)  {\it The set $K$.}  Our first step is the construction of a certain compact subset $K$ of the plane.  We start with $\overline{\D}$, 
the closed unit disc.  For each $n\in\N$,   we consider the circle $\Gamma_n$ of radius $1 +1/n$. Then we place $n$ points equally spaced on $\Gamma_n$. 
 The totality of these points is $U$, and  the union of    $U$ with  $\overline{\D}$ and with $\T$ form   the sets $K$ and $L$, respectively. Clearly $K$ and $L$ 
are compact,   $U$ is the set of isolated  points of $K$, $L=\overline{U}$, and the interior of $K$ is the open disc $\D$. \s
 
(2)  {\it The algebra $B$.} We now define a  natural Banach function algebra $B$ of quasi-analytic functions on $\overline{\D}$. For the general theory 
of such Banach function algebras,  see \cite[\S4.4]{D}; the specific example that we require is given in \cite{DM}.   Thus $(B, \norm_B) $ is an algebra of  
infinitely-differentiable functions on  $\overline{\D}$ such that    $\lV g \rV_B< \infty$, where $\norm_B$ is  specified by the formula
$$
 \lV g \rV_B = \sum_{k=0}^\infty \frac{1}{M_k}\lv  g^{(k)}\rv_{\overline{\D}}\quad (g \in B)
 $$ 
 for a suitable sequence $(M_k : k\in \Z^+)$.    The norm is chosen so that $(B, \norm_B)$ is a  natural Banach function on $\overline{\D}$ and such that the algebra $B$ 
is quasi-analytic, in the sense  that $g=0$ whenever $g\in B$ is such that $g^{(j)}(z_0) =0\,\;(j\in \Z^+)$ for any point $z_0 \in  \overline{\D}$.
 [In fact, we can take $M_k= k! \log 2 \cdots \log(k+2)$ for $k \in \N$, with $M_0=1$, for example.]

We note that the uniform closure of the algebra $B$ is $A(\overline{\D})$, the standard disc algebra,   a natural uniform algebra on $\overline{\D}$.\s
 
(3)  {\it The algebra $C$.} Next, we consider the Banach function algebra $C = {\rm Lip\,}L$  of Lipschitz functions on $L$. These Lipschitz algebras
 are also discussed  in  \cite[\S4.4]{D}.   Thus $C$ consists of the continuous functions on $L$ such that  $\norm_C< \infty$, where $\norm_C$ is  specified by the formula
$$ 
 \lV h \rV_C = \lv h \rv_L + \sup\left\{ \frac{\lv f(z)-f(w)\rv}{\lv z-w\rv} : z,w \in L,\,z\neq w \right\}\,.
$$ 
It is shown in  \cite[Theorem 4.4.24]{D}  that $C$ is a natural Banach function algebra on $L$.  Clearly the uniform closure of $C$ is $C(L)$.\s

(4)  {\it The algebra $A$.} We now form a   Banach function algebra $A$ which is the combination of $B$ and $C$. Indeed,
$$
A = \{ f \in C(K): f\mid  \overline{\D} \in B,\, f\mid L \in C\}\,.
$$
The norm $\norm_A$ on $A$ is  specified by
$$
\lV f \rV_A = \max\{\lV f\mid  \overline{\D}\rV_B,\,\lV f\mid L\rV_C\} \quad (f \in A)\,.
$$ 
We see that $(A,\norm_A)$ is a Banach function algebra on $K$.  \s

(5)  {\it The algebra $\frak A$.} We next look at the functions $f\in A$  such that  
\begin{equation}\label{(3.1)}
\lim_{n\to \infty}  \frac{f(x_n)-f(z_0)}{x_n-z_0} =f'(z_0)
\end{equation}
whenever $z_0\in\T$ and $(x_n)$ is a sequence in $U$ with $\lim_{n\to \infty}x_n =z_0$.     These functions $f$ clearly form   a subalgebra, say $\mathfrak A$, of $A$.

We {\it claim\/} that $\mathfrak A$ is a closed subalgebra of $A$. Indeed, take a sequence $(f_k)$  in $\mathfrak A$ such that $f_k\to f $  in $(A, \norm_A)$. 
 To see that $f \in  \mathfrak A$, take $z_0\in\T$ and a sequence  $(x_n)$ in $U$ with $\lim_{n\to \infty}x_n =z_0$.  Fix $\varepsilon >0$.
   Since $\lV f_k\mid \overline{\D} -f\mid \overline{\D}\rV_B \to 0$, there exists $k_1 \in \N$ such that  
$$
\lv f_k'(z_0)- f'(z_0)\rv <\varepsilon\quad(k\geq k_1)\,.
   $$
 Since $\lV f_k\mid L -f\mid L\rV_C\to 0$, there exists $k_2 \in \N$ such that
 $$
 \frac{\lv (f_k-f)(x_n) - (f_k-f)(z_0)\rv}{\lv x_n-z_0\rv}  <\varepsilon\quad(k\geq k_2)\,.
 $$
Set $k_0=\max\{k_1,k_2\}$. Then there exists $n_0 \in \N$ such that 
$$
\lv   \frac{f_{k_0}(x_n)-f_{k_0}(z_0)}{x_n-z_0} - f'_{k_0}(z_0)\rv <  \varepsilon\quad (n\geq n_0)\,.
$$
Hence
$$
\lv   \frac{f(x_n)-f(z_0)}{x_n-z_0} - f'(z_0)\rv < 3\varepsilon\quad (n\geq n_0)\,.
$$
It follows that $\lim_{n\to \infty}   ({f(x_n)-f(z_0))}/(x_n-z_0) =f'(z_0)$, and so $f \in\mathfrak A$, giving the claim.

Since  $\mathfrak A$ contains the restrictions to $K$ of the polynomials, it is clear that  $\mathfrak A$ contains the constants and separates the points of $K$, 
and so $\mathfrak A$ is a Banach function algebra on $K$.    Since $\mathfrak A$ also contains   $\delta_z$ for  each $z \in U$,
 it is easy to see that the uniform closure  of $\mathfrak A$ is $A(K)$,  the algebra of all continuous functions on $K$ that are analytic on ${\rm int} K= \D$.\s

(6)  {\it The naturality of $\mathfrak A$.}  We next {\it claim\/} that $\mathfrak A$ is natural on $K$.  It is  a general result   
that it suffices to prove: (i)   for each $f \in \mathfrak A$ such that ${\bf Z}(f) =\{z\in L : f(z)=0\}$ is void, it follows that $1/f \in \mathfrak A$; 
(ii) the uniform closure $A(K)$ of  $\mathfrak A$ is natural on $K$. (This follows immediately from \cite[Proposition 4.1.5(i)]{D}.) 
Clause (ii) is a standard result of Arens \cite[Theorem 4.3.14]{D}.   For (i), take $f \in \mathfrak A$  with ${\bf Z}(f) =\emptyset$,  and set $g =1/f \in C(K)$.
 Since $B$ and $C$ are each natural, it follows that  $g\mid \overline{\D} \in B$ and $g\mid L \in C$, and hence $g\in A$. It is clear that $g$ 
satisfies equation (\ref{(3.1)}), and so $g\in \mathfrak A$.   Thus $\mathfrak A$ is natural on $K$.\s

(6)  {\it The peak points  of $\mathfrak A$.} 
Certainly each point of $U$ is an isolated peak point for $A$.

We now {\it claim} that  there are no other peak points.  It is enough to show that the point $z=1$ is not a peak point.

Assume towards a contradiction that $f\in \frak A$ and that $f$ peaks at $1$, say with $f(1)=1$. Then $f\mid \overline{\D} =1 +g$ for a certain function $g\in B$. 
 The function $g$ is not zero, and so, since $B$ is a quasi-analytic algebra, there exists  $k\in \N$ such that $g^{(k)} (1)\neq 0$, where the
 derivative is calculated with  respect to $\overline{\D}\/$;  we take $k_0\in \N$ to be the minimum such $k$.
 
 First, suppose that $k_0\geq 2$.  Then there exists $\alpha \in \C\setminus \{0\}$ such that
$$
f(z) =1 + \alpha (z-1)^{k_0} + o(\lv z-1\rv^{k_0})
$$ 
as  $z\to 1$ with $z\in \overline{\D}$.  But, whatever the value of $\alpha$, we can choose a sequence  $(z_n)$  in $\overline{\D}$ with 
$\lim_{n\to\infty} z_n =1$ and $\Re (\alpha (z_n-1)^{k_0})  >0$  for each $n\in\N$.  This implies that  $\lv f(z_n)\rv > 1$ for all sufficiently 
large $n\in\N$,  a contradiction.

 Second, suppose that $k_0= 1$. We must now use points outside  $\overline{\D}$, and so, at this point, we regard $g \in \frak A$ as a function on $K$ with $f=1+g$. 
We know that   there exists $\alpha \in \C\setminus \{0\}$ such that
 $$
f(z) =1 + \alpha (z-1)  + o(\lv z-1\rv )
$$
as  $z\to 1$ with $z\in \overline{\D}$. Now the set $\{z \in  \C : \Re (\alpha (z-1))\geq 0\}$  is  a half-plane with
 $1$ on the boundary line.  In the case where $\alpha\not\in \R^+$, we can choose a sequence $(z_n)$  in $\overline{\D}$ as before to obtain a contradiction.  
Thus we may suppose that  $\alpha \in \R^+$. But now, by the choice of the set $U$,  there is a sequence $(x_n)$  in $U$ with $\Re (\alpha (x_n-1)) >0$  for each $n\in\N$.
 By the construction of our algebra $\frak A$,  we know that 
$$
\lim_{n\to \infty} \Re\left( \frac{f(x_n)-f(1)}{x_n-1} \right)=\alpha > 0\,.
$$
Thus   $\Re f(x_n) > 1$ for all sufficiently large $n\in\N$. It follows that $\lv f(x_n)\rv> 1$ for all sufficiently large $n\in\N$, and this is 
the required contradiction.\s

(7) {\it The conclusion.}  We have shown that $\mathfrak A$ is a natural Banach function algebra on $K$ such that $S_0({\mathfrak A})=U$, 
the countable set of isolated points of $K$.  Let $z \in U$, with corresponding maximal ideal $M_z$. Then $M_z$ is singly generated  by the function 
$1-\delta_z$, and so all maximal ideals corresponding to points of $\Gamma_0({\mathfrak A})$ are singly generated. However, ${\mathfrak A}$ is not
 a finite-dimensional algebra.}\qed
 \end{example} 

\newcommand{\email}{\texttt}

\noindent   H.\ Garth\ Dales, \\
 Department of  Mathematics and Statistics\\
Fylde College\\
University of Lancaster\\
Lancaster LA1 4YF\\
United Kingdom\\ 
\email{g.dales@lancaster.ac.uk}

\medskip

\noindent W.\ \.Zelazko,\\
Mathematical Institute, Polish Academy of Sciences\\ 
\'Sniadeckich 8, P.O. Box 21\\ 
00-956 Warsaw, Poland\\
\email{W.Zelazko@impan.pl}


\medskip



\begin{thebibliography}{99}


\bibitem{Boudi}  N.\ Boudi, {Banach algebras in which every left ideal is countably generated}, \emph{Irish Math.\ Soc. Bulletin}  48 (2002), 17--24.\vskip3mm

\bibitem{CK} R.\ Choukri and A.\ El Kinani,  {Topological algebras with ascending and descending chain conditions}, 
\emph{Acta Math.} (\emph{Basel}\/) 72 (1999), 438--443.\vskip3mm

\bibitem{CKO} R.\ Choukri, A.\ El Kinani,  and M.\ Oudadess,
{Alg\`ebres topologiques \`a id\'eaux \`a gauche ferm\'es}, \emph{Studia Mathematica} 168 (2005), 159--164.\vskip3mm

\bibitem{D} H.\ G.\ Dales, \emph{Banach algebras and automatic continuity}, London Math. Soc. Monographs, Volume 24, Clarendon Press, Oxford, 2000.\vskip3mm

\bibitem{DKKKL}  H.\ G.\ Dales, T.\ Kania, T.\ Kochanek, P.\ Koszmider, and 
N.\ J.\ Laustsen, {Maximal left ideals of the Banach algebra of bounded operators on a Banach space}, in preparation.\vskip3mm

 \bibitem{DM}    H.\ G.\ Dales  and J.\ P.\ McClure,    {Completion of normed algebras of polynomials},    \emph{Journal of the
Australian Mathematical Society} 20  (1975), 504--510.  \vskip3mm


\bibitem{FT} A.\ V.\  Ferreira and G.\  Tomassini, {Finiteness properties of topological algebras}, \emph{Ann.\  Scuola Norm.\ Sup.\ Pisa} 5 (1978), 471--488.\vskip3mm


\bibitem{GRe} H. Grauert and R. Remmert, \emph{Analytische Stellenalgebren}, Springer--Verlag, Berlin, Heidelberg, and New York, 1971.\vskip3mm

\bibitem{GRo} R. C. Gunning and H. Rossi, \emph{Analytic functions of several complex
variables}, Prentice-Hall, Englewood Cliffs, New Jersey, 1965. \vskip3mm

\bibitem{Hun} T.\ W. Hungerford,  \emph{Algebra}, Graduate Texts in Mathematics, Volume  73, Springer-Verlag, New York, Heidelberg, Berlin, 1974.\vskip3mm

\bibitem{J}  N.\ Jacobson, \emph{Basic Algebra, II}, W.\ H.\ Freeman, San Francisco, California, 1980. \vskip3mm

  

\bibitem{ST} A.\ M.\ Sinclair and A.\ W.\ Tullo, {Noetherian Banach algebras are finite dimensional}, \emph{Math.\ Annalen} 211 (1974), 151--153.\vskip3mm

\bibitem{Stout}  E.\ L.\ Stout, \emph{The theory of uniform algebras}, Bogden and Quigley, Tarrytown-on-Hudson, New York, 1971.\vskip3mm

\bibitem{Z1}  W.\ \.Zelazko, {A characterization of commutative Fr\'echet algebras with all ideals closed}, \emph{Studia Mathematica} 138
(2000), 293--300.\vskip3mm

\bibitem{Z2} W.\ \.Zelazko, {A characterization of $F$-algebras with all one-sided ideals closed}, \emph{Studia Mathematica} 168 (2005), 135--145.\vskip3mm

\bibitem{Z3} W.\ \.Zelazko, {When a closed ideal of a commutative Banach algebra contains a dense ideal}, \emph{Functiones et Approximatio, Commentarii Mathematici},  
44   (2011), 285--287.\vskip3mm
 

\bibitem{Z4} W.\ \.Zelazko,  {Concerning dense subideals in commutative Banach algebras}, \emph{Functiones et Approximatio, Commentarii Mathematici}, to appear.  
\end{thebibliography}
\end{document}